\documentclass[a4paper,reqno,11pt]{amsart}

\usepackage{amsmath}
\usepackage{amssymb}
\usepackage{amsthm}
\usepackage{latexsym}
\usepackage[all]{xypic}
\usepackage{amsfonts}
\usepackage{mathrsfs}
\usepackage{graphicx,color}
\usepackage{xcolor}
\usepackage{tikz}
\usepackage{graphics}
\usetikzlibrary[shapes]
\xyoption{all}
\usepackage[colorlinks=true,linkcolor=blue,citecolor=blue]{hyperref}
\usepackage{anysize,hyperref}
\marginsize{30mm}{30mm}{23mm}{23mm} 

\usepackage{float}
\usepackage{soul}
\setstcolor{red}

\newtheorem{theorem}{Theorem}[section]
\newtheorem{lemma}[theorem]{Lemma}

\newtheorem{proposition}[theorem]{Proposition}
\theoremstyle{definition}
\newtheorem{definition}[theorem]{Definition}
\newtheorem{definition-proposition}[theorem]{Definition-Proposition}
\newtheorem{example}[theorem]{Example}

\def\C{\mathcal{C}}
\def\D{\mathcal{D}}
\def\H{\mathcal {H}}

\def\F{\mathcal{F}}

\def \text{\mbox}

\providecommand{\ind}{\mathop{\rm ind}\nolimits}%
\providecommand{\Hom}{\mathop{\rm Hom}\nolimits}%
\renewcommand{\mod}{\mathop{\rm mod}\nolimits}%
\providecommand{\dim}{\mathop{\rm dim}\nolimits}%
\providecommand{\Obj}{\mathop{\rm Obj}\nolimits}%

\usetikzlibrary{arrows,decorations.pathmorphing,decorations.pathreplacing,positioning,shapes.geometric,shapes.misc,decorations.markings,decorations.fractals,calc,patterns}
\begin{document}

\title{Grothendieck groups of repetitive cluster categories}

\author{Huimin Chang}
\address{
Department of Applied Mathematics,
The Open University of China,
100039 Beijing,
P. R. China
}
\email{changhm@ouchn.edu.cn}

\author{Dave Murphy}
\address{
Department of Computer Science\\
	University of Verona\\
	Strada le Grazie 15\\
	37134, Verona\\
	Italy}
\email{davidjordan.murphy@univr.it}

\author[Panyue Zhou]{Panyue Zhou$^\ast$}
\address{School of Mathematics and Statistics, Changsha University of Science and Technology, 410114 Changsha, Hunan,  P. R. China}
\email{panyuezhou@163.com (P.\hspace{1mm}Zhou)}

\thanks{$^\ast$Corresponding author.}
\begin{abstract}
 In order to study cluster-tilted algebras and their intermediate coverings, Zhu introduced the notion of repetitive cluster categories, defined as the orbit categories $\D^b(\H)/\langle(\tau^{-1}\Sigma)^p\rangle$ for $1\leq p\in\mathbb{N}$, where $\H$ is a hereditary abelian category with tilting objects. In this paper, we compute partial but essential results on the Grothendieck groups of the repetitive cluster categories $\D^b(\mod KA_n)/\langle(\tau^{-1}\Sigma)^p\rangle$ and $\D^b(\mod KD_n)/\langle(\tau^{-1}\Sigma)^p\rangle$. Our results extend the known computations for classical cluster categories, reveal new structural patterns arising from the repetitive parameter $p$, and provide further evidence of the close interplay between Grothendieck groups, Auslander-Reiten theory, and Coxeter transformations.
\end{abstract}



\thanks{Huimin Chang is supported by the National Natural Science Foundation of China (Grant No.12301047). Dave Murphy is partially supported by the project LAVIE: Large Views of Smal l Phenomena: Decompositions, Localizations, and Representation Type, FIS00001706, funded by Fondo Italiano per la Scienza FIS-2021.
Panyue Zhou is supported by the National Natural Science Foundation of China (Grant No. 12371034) and by the Scientific Research Fund of Hunan Provincial Education Department (Grant No. 24A0221).}
\maketitle

\section{Introduction}
Cluster categories were introduced by Buan, Marsh, Reineke, Reiten, and Todorov \cite{BMRRT} as a categorical model for the cluster algebras of Fomin and Zelevinsky. For Dynkin type $A$, cluster categories were independently introduced by Caldero, Chapoton, and Schiffler \cite{CCS}, who provided an explicit combinatorial model based on the diagonals of a regular polygon for type $A_n$. In general, cluster categories are defined as the orbit categories $\D^b(H)/\langle F \rangle$, where $\D^b(H)$ is the bounded derived category of a hereditary algebra $H$, and $F = \tau^{-1}\Sigma$ is an automorphism generated by the composition of the inverse Auslander-Reiten translation $\tau^{-1}$ and the shift functor $\Sigma$ of $\D^b(H)$.

To study cluster-tilted algebras and their intermediate coverings, Zhu \cite{Z} introduced \emph{repetitive cluster categories}, defined as the orbit categories $\D^b(\H)/\langle(\tau^{-1}\Sigma)^p\rangle$ for $1 \leq p \in \mathbb{N}$, where $\H$ is a hereditary abelian category with tilting objects. Subsequently, Lamberti \cite{L} provided a geometric model for
$$\C_{n,\hspace{0.5mm}p}^A := \D^b(\mod KA_n)/\langle(\tau^{-1}\Sigma)^p\rangle$$ by interpreting it via diagonals in the so-called repetitive polygon $\Pi^p$. This construction generalizes the model of Caldero, Chapoton, and Schiffler for the case $p=1$ \cite{CCS}.
More recently, Gubitosi \cite{G} extended this geometric approach to type $D_n$ by showing that the repetitive cluster category
$$\C_{n,\hspace{0.5mm}p}^D := \D^b(\mod KD_n)/\langle(\tau^{-1}\Sigma)^p\rangle$$ (for $1 \leq p \in \mathbb{N}$) is equivalent to a category defined on a subset of tagged edges in a regular punctured polygon, thereby generalizing Schiffler's construction \cite{S} for $p=1$.
These developments provide deeper insights into the geometric realizations of repetitive cluster categories. A distinctive feature of these categories is that they possess a fractional Calabi-Yau dimension $\frac{2p}{q}$, which cannot be simplified. This fractional dimension marks a key difference between the categories $\C_{n,p}^A$ and $\C_{n,p}^D$ for general $p$ and their classical counterparts when $p=1$, reflecting subtle changes in both their geometric configurations and algebraic properties as $p$ increases. The fractional nature of the Calabi-Yau dimension arises from the periodic action of the automorphism $(\tau^{-1}\Sigma)^p$ and plays a crucial role in understanding the homological structure of these categories.

The Grothendieck group is an important numerical invariant associated with a category.
In a triangulated category $\C$, the isomorphism class of an object $X$ is denoted by $[X]$, and the free abelian group generated by all such isomorphism classes $[X]$ is denoted by $F$.
The split Grothendieck group $K_0^{\mathrm{sp}}(\C)$ is defined as the quotient of $F$ by the subgroup generated by all elements of the form $[X \oplus Y] - [X] - [Y]$, where $X$ and $Y$ are objects in $\C$.
Meanwhile, the Grothendieck group $K_0(\C)$ is defined as the quotient of $K_0^{\mathrm{sp}}(\C)$ by the subgroup generated by elements of the form $[\delta] = [X] + [Z] - [Y]$, where $\delta$ is a distinguished triangle of the form
  $X\rightarrow Y\rightarrow Z\rightarrow \Sigma X.$

The Grothendieck groups of triangulated categories have been extensively studied.
Barot, Kussin, and Lenzing \cite{BKL} investigated the Grothendieck group of cluster categories by analyzing the action of the Coxeter transformation on their derived categories.
Xiao and Zhu \cite{XZ} examined the relations in the Grothendieck group of a triangulated category of finite type and showed that these relations are generated by all Auslander-Reiten triangles.
Palu \cite{P} computed the Grothendieck groups for certain $2$-Calabi-Yau triangulated categories and established a generalization of the Fomin-Zelevinsky mutation rule.
Fedele \cite{F} extended these results to higher-dimensional cases, generalizing the work of Xiao and Zhu as well as Palu.
More recently, Murphy \cite{M} computed the Grothendieck groups for each family of discrete cluster categories of Dynkin type $A_{\infty}$, originally introduced by Igusa and Todorov \cite{IT}.

In this paper, we
compute partial but essential results on the Grothendieck groups of $\C_{n,p}^A$ and $\C_{n,p}^D$.
These categories arise as repetitive analogues of cluster categories of Dynkin type $A_n$ and $D_n$, and their Grothendieck groups provide important algebraic invariants encoding the homological and combinatorial structure of the categories.
Our results generalize previous computations for classical cluster categories and reveal new patterns in the interaction between Auslander–Reiten theory, triangulated structures, and Coxeter transformations.

The paper is structured as follows.
In Section 2, we review the necessary background on repetitive cluster categories of type $A_n$ and $D_n$, including their construction, fundamental properties, and the definition of the Grothendieck group. This section also recalls key results from the literature that motivate our approach and situates our work in the context of earlier computations for classical cluster categories.
Section 3 contains the proofs of our main results, where we explicitly compute the Grothendieck groups of $\C_{n,p}^A$ and $\C_{n,p}^D$ and analyze how their structure reflects both the repetitive parameter $p$ and the combinatorics of the underlying Dynkin types.
These results extend previous computations in the finite case and highlight new phenomena specific to the repetitive setting.


\subsection*{Conventions}
In this paper, we assume that $K$ denotes an algebraically closed field. We also assume that every subcategory of a category is full and closed under taking isomorphisms, finite direct sums, and direct summands. Hence, any subcategory can be determined by its indecomposable objects.
For  a triangulated category $\C$, we use $\ind\C$ to denote the set of indecomposable objects of $\C$,  $\Sigma$ to denote the shift functor of $\C$, and $\tau$  to denote the Auslander-Reiten translation of $\C$.

\section{Preliminaries}
Let $\C$ be a triangulated category and $F\colon \C\rightarrow \C$ be an autoequivalence. The orbit category $\mathcal O_{F}:= \C/\langle F\rangle$ has the same objects as $\C$ and its morphisms from $X$ to $Y$ are in bijection with
$$\bigoplus_{i\in\ \mathbb{Z}} \mathrm{Hom}_{\C}(X,F^{i}Y).$$
In general, the orbit category is not triangulated \cite{K}. However, there is a sufficient set of conditions such that the orbit category is triangulated. Let $\C=D^b(\H)$ be the bounded derived category of a hereditary abelian $K$-category $\H$, and $F\colon \C\rightarrow \C$ be an autoequivalence.  Assume that the following hypotheses hold:
\begin{itemize}
  \item [(g1)] For each indecomposable $U$ of $\H$, only finitely many objects $F^iU, i\in\mathbb{Z}$,
lie in $\H$.
  \item [(g2)] There is an integer $N\geq 0$ such that the $F$-orbit of each indecomposable of
$\C$ contains an object $\Sigma^n U$, for some $0\leq n\leq N$ and some indecomposable object $U$ of $\H$.
\end{itemize}
Then the orbit category $\C/\langle F\rangle$ admits a natural triangulated structure such that
the projection functor $\pi\colon\C\rightarrow\C/\langle F\rangle$ is triangulated. In particular, the orbit category $\C/\langle F\rangle$ is called a cluster category when $F=\tau^{-1}\Sigma$, denoted by $\C_{n}$, and the orbit category $\C/\langle F\rangle$ is called a repetitive cluster category when $F=(\tau^{-1}\Sigma)^p$, denoted by $\C_{n,p}$.

The class of objects in $\C_{n,p}$ is the same as the class of objects in $\C$ and the space of morphisms is given by
$$\Hom_{\C_{n,p}}(X,Y)=\bigoplus_{i\in\ \mathbb{Z}} \Hom_{\C}(X,(\tau^{-p}\Sigma^p)^{i}Y).$$
Observe that when $p=1$, one gets the classical cluster category $\C_{n}$. Furthermore, there is a natural projection functor $\pi_p\colon\C\rightarrow\C_{n,p}$, and the projection functor $\pi_p$ is simply denoted by $\pi$ if $p=1$. Moreover, one can define the projection functor $\eta_p\colon\C_{n,p}\rightarrow\C_{n}$ which sends an object $X$ in $\C_{n,p}$ to an object $X$ in $\C_{n}$, and $\phi\colon X\rightarrow Y$ in $\C_{n,p}$ to the morphism $\phi\colon X\rightarrow Y$ in $\C_{n}$. It is easy to check that $\pi=\eta_p\circ\pi_p$, i.e. the following diagram commutes:
$$
{
\xymatrix@-7mm@C-0.01cm{
     &&\C \ar[rdd]_{\pi_p}\ar[rr]^{\pi} & & \C_{n} \\
 \\
 &&&\C_{n,p} \ar[ruu]_{\eta_p} \\
 \\
}
}
$$
Thus, the repetitive cluster categories $\C_{n,p}$ serve as intermediate categories between the cluster categories $\C_{n}$ and derived categories $\C$.
\vspace{1mm}

We recall the concepts of the Serre functor and fractional Calabi-Yau triangulated categories, which play a central role in the study of orbit categories and repetitive cluster categories.

\begin{definition}
A $K$-linear triangulated category $\C$ has a Serre functor if it is equipped with an auto-equivalence $\nu\colon \C\rightarrow \C$ together with bifunctorial isomorphisms
$$D\Hom_{\C}(X,Y)\cong\Hom_{\C}(Y,\nu X),$$
for each $X,Y\in\C$. Moreover, if $\C$ admits a Serre functor, we say that $\C$ has Serre duality.
\end{definition}
Note that for a hereditary abelian category $\H$, if $\H$ has tilting objects, then $\H$ has Serre duality, and also Auslander-Reiten translate $\tau$ (shortly AR-translate) \cite{HRS}. Thus, a Serre functor $\nu$ of $\C$ exists, it is unique up to isomorphism and $\nu=\tau\Sigma$.
\begin{definition}
A triangulated category $\C$ with Serre functor $\nu$ is called $\frac{m}{n}$-Calabi-Yau ($\frac{m}{n}$-CY) for $n,m>0$ provided there is an isomorphism of triangle functors
$$\nu^n\cong \Sigma^m$$
\end{definition}
Note that a $\frac{m}{n}$-CY category is also  $\frac{mt}{nt}$-CY, $t\in\mathbb{Z}$. However, the converse is not always true, this means the fraction cannot be simplified in general.

In the following, we summarize some known facts about $\C_{n,p}$ from \cite{L}.
\begin{proposition} \cite[Proposition 3.3]{Z} and \cite[Lemma 2.7]{L} The following statements hold for $\C_{n,p}$.
\begin{enumerate}
\item The projection functors $\pi_p\colon\C\rightarrow\C_{n,p}$ and $\eta_p\colon\C_{n,p}\rightarrow\C_{n}$ are triangle functors.
\item The category $\C_{n,p}$ is triangulated with Serre functor $\tau\Sigma$ induced from $\D$.
\item The category $\C_{n,p}$ is a Krull-Schmidt and $\frac{2p}{p}$-CY category, i.e.
$$(\tau\Sigma)^p\cong \Sigma^{2p}.$$
\item  $\ind(\C_{n,p})=\bigcup_{i=1}^{p}(\ind(F^i\C_n))$, where $F=\tau^{-1}\Sigma$.
\end{enumerate}
\end{proposition}

\subsection{Repetitive cluster categories of type $A_{n}$}
Let $Q$ be a quiver of underlying Dynkin type $A_n$. We denote by the vertices of $Q$ with $1,2,\cdots, n$ and the arrows are $i\rightarrow i+1$ ($i=1, \cdots , n-1$), see Figure~\ref{quiverAn}.
\vspace{-2mm}

\begin{figure}[H]
\begin{center}
\[\xymatrix@R0pt@C14pt{
 n
& (n-1)\ar[l]
&\cdots \ar[l]
&2\ar[l]
&1\ar[l].
}
\]
\caption{The quiver of type  $A_n$}
\label{quiverAn}
\end{center}
\end{figure}
\vspace{-2mm}

Let $\C_A=D^b(\mod KA_n)$ be the bounded derived category of $K$-category $\mod KA_n$. Then $\C_{n,p}^A=D^b(\mod KA_n)/\langle(\tau^{-1}\Sigma)^p\rangle$ is the repetitive cluster categories of type $A_n$. When $n=1$, $\C_{n}^A=D^b(\mod KA_n)/\langle(\tau^{-1}\Sigma)\rangle$ is the cluster categories of type $A_n$.
Let $\F=\F_1$ be the fundamental domain for the functor $\tau^{-1}\Sigma$ in $\C_A$ given by the isoclasses of indecomposable objects in $\C_n^A$, and $\F_i$ be the ($\tau^{-1}\Sigma)^{i-1}$-shifts of $\F$ with $1\leq i\leq p$. We can draw the fundamental domain for the functor $\tau^{-p}\Sigma^p$ as in Figure \ref{01}.
\begin{figure}[H]
\begin{center}
\begin{tikzpicture}
\draw (-3,-1)-- (0,-1)--(-1,1)--(-2,1)--(-3,-1);
\draw (-0.8,1)--(2.2,1)--(1.2,-1)--(0.2,-1)--(-0.8,1);
\node at (2.4,0){$\cdots$};
\node at (-1.5,0){$\F$};
\node at (0.7,0){$\F_2$};
\node at (4.2,0){$\F_{p-1}$};
\node at (6.4,0){$\F_{p}$};
\draw (2.7,1)--(5.7,1)--(4.7,-1)--(3.7,-1)--(2.7,1);
\draw (4.9,-1)-- (7.9,-1)--(6.9,1)--(5.9,1)--(4.9,-1);
\end{tikzpicture}
\end{center}
\caption{Partition of the fundamental domain of $\tau^{-p}\Sigma^p$ for an odd value $p$}
\label{01}
\end{figure}

\begin{example}
Let $n=3,p=3$. We draw the AR-quiver of $\C_{3,3}^A$ as shown in Figure \ref{03},  which also can be found in \cite[Example 3.9]{C}. Every indecomposable object in the AR-quiver of $\C_{3,3}^A$ is denoted by a so-called ``diagonal", related to the geometric model of $\C_{3,3}^A$, see Section 3 in \cite{C} for details.
\begin{figure}[h]
\begin{tikzpicture}[scale=0.8,
fl/.style={->,shorten <=6pt, shorten >=6pt,>=latex}]
\coordinate (13) at (0,0) ;
\coordinate (14) at (1,1) ;
\coordinate (15) at (2,2) ;
\coordinate (24) at (2,0) ;
\coordinate (25) at (3,1) ;
\coordinate (26) at (4,2) ;
\coordinate (35) at (4,0) ;
\coordinate (36) at (5,1) ;
\coordinate (37) at (6,2) ;
\coordinate (46) at (6,0) ;
\coordinate (47) at (7,1) ;
\coordinate (48) at (8,2) ;
\coordinate (57) at (8,0) ;
\coordinate (58) at (9,1) ;
\coordinate (59) at (10,2) ;
\coordinate (68) at (10,0) ;
\coordinate (69) at (11,1) ;
\coordinate (610) at (12,2) ;
\coordinate (79) at (12,0) ;
\coordinate (710) at (13,1) ;
\coordinate (711) at (14,2) ;
\coordinate (810) at (14,0) ;
\coordinate (811) at (15,1) ;
\coordinate (812) at (16,2) ;
\coordinate (911) at (16,0) ;
\coordinate (912) at (17,1) ;
\coordinate (913) at (18,2) ;
\coordinate (1012) at (18,0) ;
\coordinate (1013) at (19,1) ;
\coordinate (1014) at (20,2) ;
\coordinate (1113) at (20,0) ;
\coordinate (1114) at (21,1) ;
\coordinate (1115) at (22,2) ;

\draw[fl] (13) -- (14) ;
\draw[fl] (14) -- (15) ;
\draw[fl] (14) -- (24) ;
\draw[fl] (15) --(25) ;
\draw[fl] (24) --(25) ;
\draw[fl] (25) --(35) ;
\draw[fl] (25) --(26) ;
\draw[fl] (35) --(36) ;
\draw[fl] (36) --(37) ;
\draw[fl] (26) --(36) ;
\draw[fl] (46) --(47) ;
\draw[fl] (47) --(48) ;
\draw[fl] (37) --(47) ;
\draw[fl] (36) --(46) ;
\draw[fl] (57) --(58) ;
\draw[fl] (58) --(59) ;
\draw[fl] (48) --(58) ;
\draw[fl] (47) --(57) ;
\draw[fl] (58) --(68) ;
\draw[fl] (68) --(69) ;
\draw[fl] (69) --(610) ;
\draw[fl] (59) --(69) ;
\draw[fl] (79) --(710) ;
\draw[fl] (710) --(711) ;
\draw[fl] (610) --(710) ;
\draw[fl] (69) --(79) ;
\draw[fl] (810) --(811) ;
\draw[fl] (811) --(812) ;
\draw[fl] (711) --(811) ;
\draw[fl] (710) --(810) ;
\draw[fl] (911) --(912) ;
\draw[fl] (912) --(913) ;
\draw[fl] (812) --(912) ;
\draw[fl] (811) --(911) ;
\draw[fl] (1012) --(1013) ;
\draw[fl] (913) --(1013) ;
\draw[fl] (912) --(1012) ;
\draw[fl] (1013) --(1113) ;
\draw (13) node[scale=0.5] {(1,3,1)} ;
\draw (14) node[scale=0.5] {(1,4,1)} ;
\draw (15) node[scale=0.5] {(1,5,1)} ;
\draw (24) node[scale=0.5] {(2,4,1)} ;
\draw (25) node[scale=0.5] {(2,5,1)} ;
\draw (26) node[scale=0.5] {(2,6,1)} ;
\draw (35) node[scale=0.5] {(3,5,1)} ;
\draw (36) node[scale=0.5] {(3,6,1)} ;
\draw (37) node[scale=0.5] {(1,3,2)} ;
\draw (46) node[scale=0.5] {(4,6,1)} ;
\draw (47) node[scale=0.5] {(1,4,2)} ;
\draw (48) node[scale=0.5] {(2,4,2)} ;
\draw (57) node[scale=0.5] {(1,5,2)} ;
\draw (58) node[scale=0.5] {(2,5,2)} ;
\draw (59) node[scale=0.5] {(3,5,2)} ;
\draw (68) node[scale=0.5] {(2,6,2)} ;
\draw (69) node[scale=0.5] {(3,6,2)} ;
\draw (610) node[scale=0.5] {(4,6,2)} ;
\draw (79) node[scale=0.5] {(1,3,3)} ;
\draw (710) node[scale=0.5] {(1,4,3)} ;
\draw (711) node[scale=0.5] {(1,5,3)} ;
\draw (810) node[scale=0.5] {(2,4,3)} ;
\draw (811) node[scale=0.5] {(2,5,3)} ;
\draw (812) node[scale=0.5] {(2,6,3)} ;
\draw (911) node[scale=0.5] {(3,5,3)} ;
\draw (912) node[scale=0.5] {(3,6,3)} ;
\draw (913) node[scale=0.5] {(1,3,1)} ;
\draw (1012) node[scale=0.5] {(4,6,3)} ;
\draw (1013) node[scale=0.5] {(1,4,1)} ;
\draw (1113) node[scale=0.5] {(1,5,1)} ;
\draw[thick, dashed, blue] (-1,-0.5) -- (7,-0.5) -- (4,2.5) -- (1.5,2.5) --cycle ;
\draw[thick, dashed, blue] (8,-0.5) -- (10.5,-0.5) -- (13,2.5) -- (5,2.5) --cycle ;
\draw[thick, dashed, blue] (11,-0.5) -- (19.5,-0.5) -- (16.5,2.5) -- (14,2.5) --cycle ;
\end{tikzpicture}
\caption{The AR-quiver of $\C_{3,3}^A$. The fundamental domains $\F_1,\F_2$ and $\F_3$ are encircled by dashed lines}
\label{03}
\end{figure}
\end{example}

\subsection{Repetitive cluster categories of type $D_{n}$}
Let $Q$ be a quiver of underlying Dynkin type $D_n$. We denote by the vertices of $Q$ with $0,1,\cdots, n-1$ and the arrows are $i-1\rightarrow i$ ($i=2, \cdots , n-1$) together with $0\rightarrow 2$, see Figure~\ref{quiverdn}.
\begin{figure}[H]
\begin{center}
\[\xymatrix@R0pt@C14pt{
&&&&&& 0 \ar[dl]\\
& (n-1)
&(n-2)\ar[l]
&\cdots \ar[l]
&3\ar[l]
&2\ar[l]\\
&&&&&& 1\ar[lu].
}
\]
\caption{The quiver of type  $D_n$}
\label{quiverdn}
\end{center}
\end{figure}
\vspace{-2mm}
Let $\C_D=D^b(\mod KD_n)$ be the bounded derived category of $K$-category $\mod KD_n$. Then $\C_{n,p}^D=D^b(\mod KD_n)/\langle(\tau^{-1}\Sigma)^p\rangle$ is the repetitive cluster categories of type $D_n$. When $n=1$, $\C_{n}^D=D^b(\mod KD_n)/\langle(\tau^{-1}\Sigma)\rangle$ is the cluster categories of type $D_n$.
Let $\F=\F_1$ be the fundamental domain for the functor $\tau^{-1}\Sigma$ in $\C_D$ given by the isoclasses of indecomposable objects in $\C_n^D$, and $\F_i$ be the ($\tau^{-1}\Sigma)^{i-1}$-shifts of $\F$ with $1\leq i\leq p$. We can draw the fundamental domain for the functor $\tau^{-p}\Sigma^p$ as in Figure \ref{funddom}.
\begin{figure}[H]
  \centering
\begin{tikzpicture}[scale=1.2]

\newcommand{\symbendblock}[4]{%
  \path (0.6,0) coordinate (A)
        (2.4,0) coordinate (B)
        ({2.4 + #2},#3) coordinate (C)
        ({2.4 + #2},#4) coordinate (D)
        ({0.6 + #2},#4) coordinate (E)
        ({0.6 + #2},#3) coordinate (F);
  \filldraw[fill=blue!10, draw=blue!80, thick]
    (A) -- (B) -- (C) -- (D) -- (E) -- (F) -- cycle;
  \pgfmathsetmacro{\cx}{(0.6+2.4+(2.4+#2)+(2.4+#2)+(0.6+#2)+(0.6+#2))/6}
  \pgfmathsetmacro{\cy}{(0+0+#3+#4+#4+#3)/6}
  \node at (\cx,\cy) {#1};
}

\begin{scope}[shift={(0,0)}]
  \symbendblock{$\mathcal{F}_1$}{1.0}{-1.8}{-3.0}
\end{scope}

\begin{scope}[shift={(2.2,0)}]
  \symbendblock{$\mathcal{F}_2$}{1.0}{-1.8}{-3.0}
\end{scope}

\node at (6.5,-1.6) {\huge $\cdots$};

\begin{scope}[shift={(6.2,0)}]
  \symbendblock{~~$\mathcal{F}_{p-1}$}{1.0}{-1.8}{-3.0}
\end{scope}

\begin{scope}[shift={(8.4,0)}]
  \symbendblock{$\mathcal{F}_p$}{1.0}{-1.8}{-3.0}
\end{scope}
\end{tikzpicture}
  \caption{Partition of the fundamental domain of $\tau^{-p}\Sigma^p$}\label{funddom}
\end{figure}
\begin{example}
When $n=3$ and $p=2$, we give an AR-quiver of $\C_{3,2}^D$ as shown in Figure \ref{funddom2},  which can also be found Figure 4 in \cite{G}. Every indecomposable object in the AR-quiver of $\C_{3,2}^D$ is expressed through the geometric model of $\C_{3,2}^D$, see Section 4 in \cite{G} for details.
\begin{figure}[H]
\begin{center}

$$\hspace{-4mm}{\small\xymatrix@C=7mm@-6mm{
&(0,2)\ar[dr]\ar[ddr]\ar@{.}[rr] & &
 (1,2)\ar[dr]\ar[ddr]\ar@{.}[rr] & &
 (2,2)\ar[dr]\ar[ddr]\ar@{.}[rr] & &
 (3,2)\ar[dr]\ar[ddr]\ar@{.}[rr] & &
 (4,2)\ar[dr]\ar[ddr]\ar@{.}[rr] & &
 (5,2)\ar[dr]\ar[ddr]\ar@{.}[rr] & &
 (0,2)\ar[dr]\ar[ddr] \\
 & & (0,1)\ar[ur]\ar@{.}[rr] & &
 (1,1)\ar[ur]\ar@{.}[rr] & &
 (2,1)\ar[ur]\ar@{.}[rr] & &
 (3,1)\ar[ur]\ar@{.}[rr] & &
 (4,1)\ar[ur]\ar@{.}[rr] & &
 (5,1)\ar[ur]\ar@{.}[rr] & &
 (0,1) \\
 & & (0,0)\ar[uur]\ar@{.}[rr] & &
 (1,0)\ar[uur]\ar@{.}[rr] & &
 (2,0)\ar[uur]\ar@{.}[rr] & &
 (3,0)\ar[uur]\ar@{.}[rr] & &
 (4,0)\ar[uur]\ar@{.}[rr] & &
 (5,0)\ar[uur]\ar@{.}[rr] & &
 (0,0)
}}
$$
\caption{The AR-quiver of $\C_{3,2}^D$}
\label{funddom2}
\end{center}
\end{figure}
\end{example}
\subsection{Grothendieck group}
Now we recall the definition of Grothendieck group in a triangulated category $\C$ and some related results.

$\mathbf{Notations}$: For an object $X$ in $\C$, we use $[X]$ to denote the isoclass of $X$ and $F$ to denote the free abelian group generated by the isoclasses $[X]$ of objects $X$ in $\C$.
\begin{definition}[\cite{F,XZ}]\label{13}
By using the same notations as above, we define the split Grothendieck group of $\C$ to be
$$K_0^{sp}(\C):=F/\langle[X\oplus Y]-[X]-[Y]\rangle,$$
where $X,Y\in\C$. We define the Grothendieck group of $\C$ to be
$$K_0(\C):=K_0^{sp}(\C)/\langle[X]-[Y]+[Z]|X\rightarrow Y\rightarrow Z\rightarrow \Sigma X \mathrm{\; is\;a\;triangle\;in\;}\C\rangle.$$
\end{definition}
The relations of the Grothendieck groups of triangulated categories we considered in this paper can be deduced from AR-triangles. We recall the following definition and related result.
\begin{definition}\label{22}
A triangulated category $\C$ is said to be of finite type provided  $$\sum_{X\in\Obj\;\C}\dim_K\Hom_{\C}(X,Y)<\infty\hspace{5mm}\mbox{and}~~ \sum_{X\in\Obj\;\C}\dim_K\Hom_{\C}(Y,X)<\infty$$ for any object $Y$ in $\C$. Moreover,
 $\C$ is called a finite triangulated category if it contains only finitely many indecomposable objects up to isomorphisms.
\end{definition}
Note that a finite triangulated category is a triangulated category of finite type by Definition \ref{22}. A triangulated category of finite type has AR-triangles, and we have the following result.
\begin{lemma}\cite[Theorem 2.1]{XZ}\label{21}
Let $\C$ be a triangulated category of finite type. Then $K_0(\C)$ is the quotient of $K_0^{sp}(\C)$ by the subgroup generated by $[\delta]=[X]+[Z]-[Y]$ for all AR-triangles.
\end{lemma}
\begin{lemma}\cite[Proposition 3.5]{BKL}\label{36}
Suppose $\C_{m}^A$ is the cluster category of finite Dynkin type $A_m$ and $\C_{m}^D$ is the cluster category of finite Dynkin type $D_m$, suppose $K_0(\C_{m}^A)$ is the Grothendieck group of $\C_m^A$ and $K_0(\C_{m}^D)$ is the Grothendieck group of $\C_m^D$. Then we have
\[ K_0(\C_{m}^A)=\left\{
\begin{array}{cc}
0  &\text{if} \ m\ \text{is}\ \text{even},\\
\mathbb{Z}  &\text{if} \ m\ \text{is}\  \text{odd},
\end{array}
\right.
\]
and
\[ K_0(\C_{m}^D)=\left\{
\begin{array}{cc}
\mathbb{Z}^2  &\text{if} \ m\ \text{is}\ \text{even},\\
\mathbb{Z}  &\text{if} \ m\ \text{is}\  \text{odd}.
\end{array}
\right.
\]
\end{lemma}
\section{The Grothendieck group of $\C_{n,p}$}

In this section, let  $H_n=KQ_n$ be a finite-dimensional algebra with $Q_n$ a Dynkin quiver with $n$ vertices. Let $\mod H_n$ denote the category of finitely generated modules, and let $D^b(\mod H_n)$ be its bounded derived category with shift functor $\Sigma$ and translation functor $\tau$.

Let $\C_{n,p}=D^b(\mod H_n)/(\tau^{-1}\Sigma)^p$ be the repetitive cluster categories. When $n=1$, $\C_{n}$ is the cluster category. Let $K_0(D^b(\mod H_n))$ be the Grothendieck group of $D^b(\mod H_n)$, and let $\phi: K_0(D^b(\mod H_n))\rightarrow K_0(D^b(\mod H_n))$ be the Coxeter transformation. For any object $X\in D^b(\mod H_n)$,
$$\phi([X])=[\tau X].$$
Then $\phi^p: K_0(D^b(\mod H_n))\rightarrow K_0(D^b(\mod H_n))$ satisfies
$$\phi^p([X])=[\tau^p X].$$

This induces short exact sequences
$$K_0(D^b(\mod H_n))\xrightarrow{1+\phi^p} K_0(D^b(\mod H_n))\xrightarrow{\pi_p} K_0(\C_{n,p})\rightarrow 0,$$
when $p$ is odd, and
$$K_0(D^b(\mod H_n))\xrightarrow{1-\phi^p} K_0(D^b(\mod H_n))\xrightarrow{\pi_p} K_0(\C_{n,p})\rightarrow 0,$$
when $p$ is even, where $\pi_p$ is induced by the covering functor $\pi_p:D^b(\mod H_n)\rightarrow \C_{n,p}$.

It is easy to show $K_0(\C_{n,p})\cong \text{Coker}(1+\phi^p)=K_0(D^b(\mod H_n))/\text{Im}(1+\phi^p)$ when $p$ is odd, and $K_0(\C_{n,p})\cong \text{Coker}(1-\phi^p)=K_0(D^b(\mod H_n))/\text{Im}(1-\phi^p)$ when $p$ is even.

Let $\{[X_1], \ldots, [X_m]\}$ be a basis of $K_0(D^b(\mod H_n))$. Then it suffices to check the action of $(1 \pm \phi^p)$ on the basis elements $\{[X_1], \ldots, [X_m]\}$.
Indeed, for any $M \in K_0(D^b(\mod H_n))$, we have
\[
(1 \pm \phi^p)[M] = [M] \pm [\tau^p M] = \sum_{i=1}^m \alpha_i [X_i] \pm \sum_{i=1}^m \alpha_i[\tau^p X_i] = \sum_{i=1}^m \alpha_i (1 \pm \phi^p)[X_i],
\]
where $[M] = \sum_{i=1}^m \alpha_i [X_i]$ and $\alpha_i \in \mathbb{Z}$ for all $1 \leq i \leq m$.
\vspace{2mm}

For a basic finite-dimensional $K$-algebra $H_n$, it is shown in \cite{Gr} that $K_0(H_n)$ and $K_0(D^b(\mod H_n))$ are isomorphic, where $K_0(H_n)$ is the Grothendieck group of $H_n$.

Throughout, we write $\langle x_1, \ldots, x_m \mid f(x_1,\ldots,x_m) \rangle$ for the quotient of the free abelian group of rank $m$ by the subgroup generated by $f(x_1,\ldots,x_m)$.

\subsection{The Grothendieck group of $\C_{n,p}^A$}

Throughout this section we consider the basis of $K_0(D^b(\mod (A_n)))$ to be ${[S_1], [S_2], \ldots, [S_n]}$, where $S_i \in \mod (A_n) \subset D^b(\mod (A_n))$ corresponds to the simple module at vertex $i \in Q_0$.
Whenever we write $[S_i]$, we always take $i \in \mod (n+1)$, where $[S_0]$ is defined in Lemma~\ref{Lem: Si in Im}.

We also consider the quiver $A_n$ to be linearly orientated as in Figure \ref{quiverAn}. The Auslander-Reiten quiver of $D^b(\mod (A_n))$ is shown in Figure \ref{00}, where the simple objects in $\mod (A_n)$ lie on the thick, red line in the bottom centre.

\begin{figure}[H]
	\centering
	\begin{tikzpicture}[scale = 1.45]
		\foreach[count = \i] \txt in {1,...,5}
		\draw (-6.9 + 2*\i,0) -- (-6 + 2*\i,1.9) -- (-5.1 + 2*\i, 0) -- (-6.9 + 2*\i,0);
		
		\foreach[count = \i] \txt in {1,...,4}
		\draw (-5.9 + 2*\i, 2) -- (-4.1 + 2*\i, 2) -- (-5 + 2*\i, 0.1) -- (-5.9 + 2*\i, 2);
		
		\draw[ultra thick,red] (-0.9,0) -- (0.9,0);
		
		\node at (0,0.7) {\footnotesize $\mathcal{M}$};
		\node at (1,1.3) {\footnotesize $\Sigma \mathcal{M}$};
		\node at (2,0.7) {\footnotesize $\Sigma^2 \mathcal{M}$};
		\node at (3,1.3) {\footnotesize $\Sigma^3 \mathcal{M}$};
		\node at (4,0.7) {\footnotesize $\Sigma^4 \mathcal{M}$};
		\node at (-1,1.3) {\footnotesize $\Sigma^{-1} \mathcal{M}$};
		\node at (-2,0.7) {\footnotesize $\Sigma^{-2} \mathcal{M}$};
		\node at (-3,1.3) {\footnotesize $\Sigma^{-3} \mathcal{M}$};
		\node at (-4,0.7) {\footnotesize $\Sigma^{-4} \mathcal{M}$};
		
		\node at (-5,1) {$\cdots$};
		\node at (5,1) {$\cdots$};
	\end{tikzpicture}
	\caption{The Auslander-Reiten quiver of $D^b(\mod (A_n))$, where each component is equivalent to the Auslander-Reiten quiver of $\mathcal{M} = \mod (A_n)$.}
\label{00}
\end{figure}

\begin{lemma}\label{Lem: Si in Im}
	Suppose $0 < p \in \mathbb{Z}$.
	Then, for all $i =1,\ldots,n$, we have
$$(1 + (-1)^{p+1} \phi^p)[S_i] = [S_i] + (-1)^{p+1} [S_{i+p}],$$
where $i+p \in \mathbb{Z}/(n+1)\mathbb{Z}$, and $[S_0] = - \sum_{i=1}^{n} [S_i]$.
\end{lemma}

\begin{proof}
	When $p=1$, then $(1 + \phi)[S_i] = [S_i] + [\tau S_i]$ for all $i =1,\ldots,n$.
	However, $\tau$ acts on indecomposable objects in $D^b(\mod H_n)$ by taking them to the indecomposable object immediately to the left in the Auslander-Reiten quiver.
	Therefore, $\tau S_i \cong S_{i+1}$ for all $i = 1,\ldots,n-1$, and $\tau S_n \cong \Sigma^{-1} P_1$.
	The result when $p=1$ then follows from the fact that $[\Sigma^{-1} P_1] = \sum_{i=1}^n [\Sigma^{-1} S_i] = - \sum_{i=1}^n [S_i]$.
	
	Noting that $\tau \Sigma^{-1} P_1 \cong \Sigma^{-2} S_1$, our claim now follows via $[\tau^p S_i] = [S_{i+p}]$, where $i+p \in \mathbb{Z}/(n+1)\mathbb{Z}$, and $[S_0] = - \sum_{i=1}^{n} [S_i]$.
\end{proof}

\subsubsection{When $n$ is odd}

\begin{lemma}\label{Lem: n odd, p,q equiv}
	Let $p,q \geq 1$ such that $p \equiv q \mod (n+1)$. Then $K_0(C_{n,p}^A) \cong K_0(C_{n,q}^A)$.
\end{lemma}

\begin{proof}
	In the derived category $D^b(\mod {A_n})$, we have the natural isomorphism $\tau^{n+1} \xrightarrow{\sim} \Sigma^{-2}$.
	We note that, as $n$ is odd, then $n+1$ is even, and so $p$ is even if and only if $q$ is even.
	
	Consider the case when $p$ is even, and $p = q +a(n+1)$ for some $a \in \mathbb{Z}$.
	Then we have $(1 - \phi^p)[X] = [X] - [\tau^p X] = [X] - [\tau^{q + a(n+1)}X] = [X] - [\Sigma^{-2a} (\tau^q X)] = [X] - [\tau^q X] = (1 - \phi^q)[X]$ for all $X \in D^b(\mod {A_n})$.
	Therefore $\mathrm{Im}(1 - \phi^p) = \mathrm{Im}(1 - \phi^q)$.
	
	An analogous argument shows that when $p$ is odd, then $\mathrm{Im}(1 + \phi^p) = \mathrm{Im}(1 + \phi^q)$, and so $K_0(C_{n,p}^A) \cong K_0(C_{n,q}^A)$ whenever $p \equiv q \mod (n+1)$.
\end{proof}

\begin{lemma}\label{Lem: n odd, p equiv 0}
	Let $n$ be odd, and let $p \equiv 0 \mod (n+1)$.
	Then $K_0(C_{n,p}^A) \cong \mathbb{Z}^n$.
\end{lemma}

\begin{proof}
	As $n$ is odd, then $p = a(n+1)$ must be even.
	Then $$(1 - \phi^p) [X] = [X] - [\tau^p X] = [X] - [\tau^{a(n+1)} X] = [X] - [\Sigma^{-2a} X] =0$$ for all $X \in D^b(\mod {A_n})$.
	Hence $\mathrm{Im}(1 - \phi^p) =0$, and so $K_0(C_{n,p}^A) \cong K_0(\D^b(\mod {A_n})) \cong \mathbb{Z}^n$.
\end{proof}

\begin{proposition}\label{Prop: n odd, p+q equiv 0}
	Let $p+q \equiv 0 \mod (n+1)$. Then $K_0(C_{n,p}^A) \cong K_0(C_{n,q}^A)$.
\end{proposition}

\begin{proof}
	By Lemma \ref{Lem: Si in Im}, we have that
	\[
	\mathrm{Im}(1 +(-1)^{p+1} \phi^p) = \langle [S_i] + (-1)^{p+1} [S_{i+p}], [S_{n+1-p}] + (-1)^p \sum_{j=1}^n [S_j] \rangle.
	\]
	Analogously, we have
	\[
	\mathrm{Im}(1 +(-1)^{q+1} \phi^q) = \langle [S_i] + (-1)^{p+1} [S_{i+q}], [S_{n+1-q}] + (-1)^q \sum_{j=1}^n [S_j] \rangle.
	\]
	
	Consider the automorphism $\eta \colon K_0(D^b(\mod {A_n})) \rightarrow K_0(D^b(\mod {A_n}))$ such that $$\eta \colon [S_i] \mapsto [S_{n+1-i}].$$
	Under the automorphism $\eta$, we get
$$
\begin{aligned}
\eta([S_i] +(-1)^{p+1} [S_{i+p}]) &= [S_{n+1-i}] +(-1)^{p+1} [S_{n+1-i-p}]\\
 &= [S_l] +(-1)^{q+1} [S_{l-p}] = [S_l] +(-1)^{q+1} [S_{l+q}]
\end{aligned}
$$
where the third equality follows from the assumption that $p+q \equiv 0 \mod (n+1)$, and $l= n+1-i$.
	Finally, we can see how $\eta$ acts on the other relation
	\[
	\eta([S_{n+1-p}] +(-1)^p \sum_{j=1}^n [S_j]) = [S_{p}] +(-1)^p \sum_{j=1}^n [S_{n+1-j}] = [S_{n+1-q}] +(-1)^q \sum_{k=1}^n [S_k],
	\]
	where, again, the second equality follows from the assumption that $p+q \equiv 0 \mod (n+1)$.
	
	Thus, $\eta$ is an isomorphism between $\mathrm{Im}(1 \pm \phi^p)$ and $\mathrm{Im}(1 \pm \phi^q)$, hence  $K_0(C_{n,p}^A) \cong K_0(C_{n,q}^A)$.
\end{proof}

\begin{proposition}\label{Prop: n odd, epi to K0}
	Let $n$ be odd, suppose that $0 < p \leq \frac{n+1}{2}$ and $n+1 \equiv k \mod (p)$ with $0 < k \leq p$.
	Then there exists an epimorphism
	\[
	\langle [S_1], [S_2],\ldots, [S_m] \rangle \twoheadrightarrow K_0(C_{n,p}^A),
	\]
	where $m$ is the greatest common divisor of $p$ and $k$.
\end{proposition}

\begin{proof}
	Suppose that $p$ is even.
	By Lemma \ref{Lem: Si in Im}, it is clear that $[S_i] = [S_j]$ in $K_0(C_{n,p})$, for $i \equiv j \mod (p)$.
	This is due to the sum of the relations $([S_i] - [S_{i+p}]) + ([S_{i+p}] - [S_{i+2p}]) + \ldots + ([S_{j-p}] - [S_j]) = [S_i] - [S_j]$ in $\mathrm{Im}(1 - \phi^p)$.
	It remains to consider the following relations;
	\[
	[S_n] - [S_{p-1}], [S_{n-1}] - [S_{p-2}], \ldots, [S_{n-p+2}] - [S_1].
	\]
	
	Let $ 0 < k \leq p$, then the above relations become;
	\[
	[S_{k-1}] - [S_{p-1}], [S_{k-2}] - [S_{p-2}], \ldots, [S_{k+1}] - [S_1].
	\]
	Then we have the relation $[S_{k-1}] - [S_{p-1}]$, however, we also have the relation $[S_{l_1}] - [S_{k-1}]$, where $k-1 \equiv l_1 \mod(p)$.
	Hence we have the relation $[S_{l_1}] - [S_{p-1}]$.
	Via this method, we find all of the relations $[S_{l_b}] - [S_{p-1}]$ such that $bk -1 \equiv l_b \mod(p)$.
	It is clear that $l_b$ and $l_{b+1}$ have a difference of $k$, and so by elementary modular arithmetic, we find that we split the set $\{[S_1], [S_2], \ldots, [S_p]\}$ into $m$ subsets, each containing an exactly one of the elements $\{[S_1], [S_2], \ldots, [S_m]\}$ where $m$ is the greatest common divisor of $p$ and $k$.
	
	Now suppose that $p$ is odd.
	Then we have the alternating sum of relations due to Lemma \ref{Lem: Si in Im}, $([S_i] + [S_{i+p}]) - ([S_{i+p}] + [S_{i+2p}]) + \ldots + (-1)^r ([S_{j-p}] + [S_j]) = [S_i] + (-1)^r [S_j]$ in $\mathrm{Im}(1 + \phi^p)$, where $i + rp = j-p$.
	An analogous argument to the case when $p$ is even then proves our claim.
\end{proof}

\begin{theorem}\label{Thm: n odd, p even}
	Let $n,p,m$ as in Proposition \ref{Prop: n odd, epi to K0}, and further suppose that $p$ is even.
	Then
	\[
	K_0(C_{n,p}^A) \cong \langle [S_1], [S_2], \ldots, [S_m] \mid [S_c] + \sum_{j=1}^{m} \alpha_j [S_j] = 0 \rangle,
	\]
	where $n+1-p \equiv c \mod(m)$, and

\[ \alpha_j =\left\{
\begin{array}{ll}
a+1, & \text{\; if \;} 0<j \leq b,\\
		a, & \text{\; if \;} b < j \leq m,
\end{array}
\right.
\]	
for $am+b=n$.
\end{theorem}

\begin{proof}
	As $p$ is even, then our relations are $[S_j] - [S_{j+m}]$ by Lemma \ref{Lem: Si in Im}, and so $[S_j] = [S_{j+m}]$ in $K_0(C_{n,p}^A)$.
	This implies there are no more relations to consider except $[S_{n+1-p}] + \sum_{j=1}^{n} [S_j]$.
	Our result follows from noticing that $[S_c] = [S_{n+1-p}]$, and that $[S_j] = [S_{j+sm}]$ for all $s$ such that $0 < j+sm \leq n$.
	Therefore $0 \leq s \leq a$ if $0 < j \leq b$, and $0 \leq s < a$ if $b < j \leq m$, and so,
	\[
	[S_{n+1-p}] + \sum_{j=1}^{n} [S_j] = [S_c] + \sum_{j=1}^b \sum_{s=0}^a [S_{j+sm}] + \sum_{j=b+1}^m \sum_{s=0}^{a-1} [S_{j+sm}] = [S_c] + \sum_{j=1}^b (a+1) [S_j] + \sum_{j=b+1}^m a [S_j].
	\]
\end{proof}

\begin{theorem}\label{Thm: n odd, p odd}
	Let $n,p,m$ as in Proposition \ref{Prop: n odd, epi to K0}, and further suppose that $p \not\equiv \pm 1 \mod (n+1)$ is odd, and that $\frac{p}{m}$ is even.
	Then,
\[ K_0(C_{n,p}^A) =\left\{
\begin{array}{ll}
\langle [S_1], [S_2], \ldots, [S_m] \mid (-1)^t [S_c] - \sum_{j=1}^{b} [S_j] = 0 \rangle, & \text{if\;} $a$ \text{\;is\;even},\\
		\langle [S_1], [S_2], \ldots, [S_m] \mid (-1)^t [S_c] - \sum_{j=b+1}^{m} [S_j] = 0 \rangle, & \text{if\;} $a$ \text{\;is\;odd},
\end{array}
\right.
\]	
	where $n+1-p = tm +c$, and $am+b=n$.
	Further, if $\frac{p}{m}$ is odd.
	Then,
\[ K_0(C_{n,p}^A) =\left\{
\begin{array}{ll}
\langle [S_1], [S_2], \ldots, [S_m] \mid 2[S_j] = 0, (-1)^t [S_c] - \sum_{j=1}^{b} [S_j] = 0\rangle, & \text{if\;} $a$ \text{\;is\;even},\\
		\langle [S_1], [S_2], \ldots, [S_m] \mid 2[S_j] = 0, (-1)^t [S_c] - \sum_{j=b+1}^{m} [S_j] = 0 \rangle, & \text{if\;} $a$ \text{\;is\;odd}.
\end{array}
\right.
\]	
\end{theorem}

\begin{proof}
	As $p$ is odd, then our relations are $[S_j] + [S_{j+m}]$ by Proposition \ref{Prop: n odd, epi to K0}, and so $[S_j] = (-1)^s[S_{j+sm}]$ in $K_0(C_{n,p}^A)$.
	
	We have two conditions to consider, whether $m \mid p$ is odd or even, and if $a$ is even.
	\begin{itemize}
		\item 	If $\frac{p}{m}$ is even, then there exists an even number of elements in the equivalence class of $[S_i]$ in $\{[S_1],[S_2], \ldots, [S_p]\}$ for $0 < i \leq m$, and so $[S_i] = (-1)^s[S_{i+sm}]$, where $s$ is even if $sm=p$.
		Therefore we do not induce any more relations on $[S_i]$ than those we have already considered, as $[S_i] = [S_i]$.
		
		However, if $\frac{p}{m}$ is odd, then there exists an odd number of elements in the equivalence class of $[S_i]$ in $\{[S_1],[S_2], \ldots, [S_p]\}$ for $0 < i \leq m$, and so $[S_i] = (-1)^s[S_{i+sm}]$, where $s$ is odd if $sm=p$.
		Therefore we find the new relation $[S_i] = -[S_i]$ for all $0 < i \leq m$, and so $2[S_i]=0$.\\
		
		\item Consider the relation $[S_{n+1-p}] - \sum_{j=1}^n [S_j]$.
		Then $[S_{n+1-p}] = (-1)^t[S_c]$ for $tm + c = n+1-p$, and $\sum_{j=1}^n [S_j] = \sum_{j=1}^r [S_j]$, where $2tm + r =n$ for $0 \leq r < 2m$.
		This is due to
		\[
		\sum_{j=r+1}^{r+1+2m} [S_j] = \sum_{j=r+1}^{r+1+m} [S_j] + [S_{j+m}] = \sum_{j=r+1}^{r+1+m} [S_j] - [S_j] = 0.
		\]
		
		If $a$ is even, then $a=2t$ and $r=b < m$, and so we have the relation $[S_{n+1-p}] - \sum_{j=1}^n [S_j] = [S_{n+1-p}] - \sum_{j=1}^b [S_j] = (-1)^t[S_c] - \sum_{j=1}^b [S_j]$.
		Hence $(-1)^t[S_c] - \sum_{j=1}^b [S_j] =0$ in $K_0(C_{n,p}^A)$ if $a$ is even.
		
		If $a$ is odd, then $a=2t+1$, and $r=b+m$, so we have the relation
		\[
		[S_{n+1-p}] - \sum_{j=1}^n [S_j] = (-1)^t[S_c] - \sum_{j=1}^m [S_j] + \sum_{j=1}^b [S_j] = (-1)^t [S_c] - \sum_{j=b+1}^m [S_j].
		\]
		Hence $(-1)^t [S_c] - \sum_{j=b+1}^m [S_j] = 0$ in $K_0(C_{n,p}^A)$ if $a$ is odd.
	\end{itemize}
	
	Our claim follows by consideration of all combinations of the above two cases.
\end{proof}


\begin{lemma}\label{Lem: n odd, p equiv 1}
	Let $n$ be odd, and suppose that $p \equiv \pm 1 \mod (n+1)$, then $K_0(C_{n,p}^A) \cong \mathbb{Z}$.
\end{lemma}

\begin{proof}
	By Lemma \ref{Lem: Si in Im}, we have the set of relations
	\[
	\{[S_1] + [S_2], [S_2] + [S_3], \ldots, [S_{n-1}] + [S_n], [S_n] - \sum_{i=1}^n [S_i]\}.
	\]
	The first $n-1$ relations in the above set imply that $[S_i] = (-1)^{i+1}[S_1]$ in $K_0(C_{n,p})$.
	Therefore we get:
	\begin{align*}
		[S_n] - \sum_{i=1}^n [S_i] &= [S_n] - [S_n] - \sum_{i=1}^{\frac{n-1}{2}} ([S_{2i-1}] +[S_{2i}])\\
		&= 0- \sum_{i=1}^{\frac{n-1}{2}} ((-1)^{2i}[S_1] + (-1)^{2i+1} [S_{2i}]\\
		&=0.
	\end{align*}
	Hence, 	$K_0(C_{n,p}^A) \cong \langle [S_1] \rangle \cong \mathbb{Z}$.
\end{proof}

As a subcase of Lemma \ref{Lem: n odd, p equiv 1}, when $p=1$, we recover the Grothendieck group of the classical cluster category of type $A_n$, with $n$ odd.
This is a result originally due to Barot, Kussin and Lenzing \cite{BKL}, see Lemma \ref{36}.

\begin{example}
Suppose $n=3$. Then we have
$$K_0(C_{3,1}^A)=K_0(C_{3,5}^A)=\cdots=K_0(C_{3,1+4k}^A)=\mathbb{Z}$$
by Lemma \ref{Lem: n odd, p equiv 1},
$$K_0(C_{3,2}^A)=K_0(C_{3,6}^A)=\cdots=K_0(C_{3,2+4k}^A), $$
and
$$K_0(C_{3,3}^A)=K_0(C_{3,7}^A)=\cdots=K_0(C_{3,4+4k}^A)=\mathbb{Z}$$
by Lemma \ref{Prop: n odd, p+q equiv 0},
$$K_0(C_{3,4}^A)=K_0(C_{3,8}^A)=\cdots=K_0(C_{3,4k}^A)=\mathbb{Z}^3$$
by Lemma \ref{Lem: n odd, p equiv 0},
where $k\geq 1$ and $k\in\mathbb{Z}$.
\end{example}

\subsubsection{When $n$ is even}

\begin{lemma}\label{Lem: n even, p,q equiv}
	Let $p,q \geq 1$ such that $p \equiv q \mod (2(n+1))$. Then $K_0(C_{n,p}^A) \cong K_0(C_{n,q}^A)$.
\end{lemma}
\begin{proof}
Since $n$ is even and  $p \equiv q \mod (2(n+1))$,  $p$ is even/odd if and only if $q$ is even/odd.
	
	Consider the case when $p$ is even, and $p = q +2a(n+1)$ for some $a \in \mathbb{Z}$.
	Then we have $(1 - \phi^p)(X) = [X] - [\tau^p X] = [X] - [\tau^{q + 2a(n+1)}X] = [X] - [\Sigma^{-4a} (\tau^q X)] = [X] - [\tau^q X] = (1 - \phi^q)(X)$ for all $X \in D^b(\mod {A_n})$.
	Therefore $\mathrm{Im}(1 - \phi^p) = \mathrm{Im}(1 - \phi^q)$.
	
	An analogous argument shows that when $p$ is odd, then $\mathrm{Im}(1 + \phi^p) = \mathrm{Im}(1 + \phi^q)$, and so $K_0(C_{n,p}^A) \cong K_0(C_{n,q}^A)$ whenever $p \equiv q \mod (2(n+1))$.
\end{proof}

\begin{proposition}\label{11}
	Suppose that $p,q \geq 1$ with $p+q \equiv 0 \mod (2(n+1))$.
	Then $K_0(C_{n,p}^A) \cong K_0(C_{n,q}^A)$.
\end{proposition}
\begin{proof}
The equation $p \equiv q \mod (2(n+1))$ ensures that $p$ is even/odd if and only if $q$ is even/odd. Otherwise the proof is similar as Lemma \ref{Lem: Si in Im}.
\end{proof}

\begin{lemma}\label{Lem: n even, p equiv 0}
	Suppose $p \equiv 0 \mod (2(n+1))$, then $K_0(C_{n,p}^A) \cong \mathbb{Z}^n$.
\end{lemma}

\begin{proof}
	Consider that $p$ must be even, then we have $(1 - \phi^p) [X] = [X] - [\tau^p X] = [X] - [\tau^{2a(n+1)} X] = [X] - [\Sigma^{-4a} X] =  [X] - [X] = 0$ for all $X \in D^b(\mod (A_n))$.
	Hence $\mathrm{Im}(1 - \phi^p) =0$, and so $K_0(C_{n,p}^A) \cong K_0( D^b(\mod (A_n)) \cong \mathbb{Z}^n$.
\end{proof}

\begin{lemma}\label{Lem: n even, p equiv n+1}
	Suppose $p \equiv (n+1) \mod (2(n+1))$, then $K_0(C_{n,p}^A) \cong (\mathbb{Z}/2\mathbb{Z})^n$.
\end{lemma}

\begin{proof}
	Consider that $p$ must be odd, then we have $(1 + \phi^p) [X] = [X] + [\tau^p X] = [X] + [\tau^{a(n+1)} X] = [X] + [\Sigma^{-2a} X] = [X] + [X] =2[X]$.
	Hence
	\[
	\mathrm{Im}(1+ \phi^p) = \langle 2[S_1], 2[S_2], \ldots, 2[S_n] \rangle,
	\]
	and so $K_0(C_{n,p}^A) \cong K_0(\D^b(\mod(A_n)))/\mathrm{Im}(1 + \phi^p) \cong (\mathbb{Z}/2\mathbb{Z})^n$.
\end{proof}


\begin{lemma}\label{Lem: n even, p equiv 1}
	Suppose  $p \equiv \pm 1 \mod (2(n+1))$, then $K_0(C_{n,p}^A) = 0$.
\end{lemma}

\begin{proof}
	By Lemma \ref{Lem: Si in Im}, we have the set of relations
	\[
	\{[S_1] + [S_2], [S_2] + [S_3], \ldots, [S_{n-1}] + [S_n], [S_n] - \sum_{i=1}^n [S_i]\}.
	\]
	The first $n-1$ relations in the above set imply that $[S_i] = (-1)^{i+1}[S_1]$.
	Therefore we get that
$$
\begin{aligned}
{[S_n]-\sum_{i=1}^n [S_i]} &= [S_n] -  (\sum_{i=1}^{\frac{n-1}{2}} [S_{2i-1}] +[S_{2i}]) \\
&= [S_n] - (\sum_{i=1}^{\frac{n-1}{2}} (-1)^{2i}[S_1] + (-1)^{2i+1} [S_{2i}]) = [S_n]\\
 &= (-1)^{n+1}[S_1].
\end{aligned}
$$
	Hence,
	\[
	K_0(C_{n,p}^A) \cong \langle [S_1] \mid (-1)^{n+1}[S_1]=0 \rangle =0.
	\]
\end{proof}

As a subcase of Lemma \ref{Lem: n even, p equiv 1}, when $p=1$, we recover the Grothendieck group of the classical cluster category of type $A_n$, with $n$ even.
This is a result originally due to Barot, Kussin and Lenzing \cite{BKL}, see Lemma \ref{36}.

The proofs of the following results are analogous to their respective results when $n$ is odd.

\begin{proposition}\label{Prop: n even, epi to K0}
	Let $n$ be even, and suppose that $0 < p \leq n+1$, and
 $2(n+1) \equiv k \mod (p)$ with $0 < k \leq p$.
	Then there exists an epimorphism
	\[
	\langle [S_1], [S_2],\ldots, [S_m] \rangle \twoheadrightarrow K_0(C_{n,p}^A),
	\]
	where $m$ is the greatest common divisor of $p$ and $k$.
\end{proposition}

\begin{theorem}\label{Thm: n even, p even}
	Let $n,p,m$ as in Proposition \ref{Prop: n even, epi to K0}, and further suppose that $p$ is even.
	Then
	\[
	K_0(C_{n,p}^A) \cong \langle [S_1], [S_2], \ldots, [S_m] \mid [S_c] + \sum_{j=1}^{m} \alpha_j [S_j] = 0 \rangle,
	\]
	where $n+1 - p \equiv c \mod (m)$, and
\[ \alpha_j =\left\{
\begin{array}{ll}
a+1, & \text{if \;} 0<j \leq b,\\
		a, & \text{if \;} b < j \leq m,
\end{array}
\right.
\]	
for $am+b=n$.	
\end{theorem}

\begin{theorem}\label{Thm: n even,  p odd}
	Let $n,p,m$ as in Proposition \ref{Prop: n even, epi to K0}, and further suppose that $p \not\equiv \pm 1 \mod (2(n+1))$ is odd, and that $\frac{p}{m}$ is even.
	Then
\[ K_0(C_{n,p}^A) \cong\left\{
\begin{array}{ll}
\langle [S_1], [S_2], \ldots, [S_m] \mid (-1)^t [S_c] - \sum_{j=1}^{b} [S_j] = 0 \rangle, & \text{if\;}$a$\text{\;is\;even\;},\\
		\langle [S_1], [S_2], \ldots, [S_m] \mid (-1)^t [S_c] - \sum_{j=b+1}^{m} [S_j] = 0 \rangle, & \text{if\;}$a$\text{\;is\;odd\;},
\end{array}
\right.
\]	
	where $n+1 -p = tm +c$, and $am+b=n$.
	Further, if $\frac{p}{m}$ is odd.
	Then,
\[ K_0(C_{n,p}^A) \cong\left\{
\begin{array}{ll}
\langle [S_1], [S_2], \ldots, [S_m] \mid 2[S_j] = 0, (-1)^t [S_c] - \sum_{j=1}^{b} [S_j] = 0\rangle, & \text{if\;}$a$\text{\;is\;even\;},\\
		\langle [S_1], [S_2], \ldots, [S_m] \mid 2[S_j] = 0, (-1)^t [S_c] - \sum_{j=b+1}^{m} [S_j] = 0 \rangle, & \text{if\;}$a$\text{\;is\;odd\;}.
\end{array}
\right.
\]	
\end{theorem}
\begin{example}
Suppose $n=2$. Then we have
$$K_0(C_{2,1}^A)=K_0(C_{2,5}^A)=K_0(C_{2,7}^A)=0$$
by Lemma \ref{Lem: n even, p equiv 1} and Lemma \ref{11},
$$K_0(C_{2,2}^A)=K_0(C_{2,4}^A)=K_0(C_{2,8}^A)$$
and
$$K_0(C_{2,3}^A)=(\mathbb{Z}/2\mathbb{Z})^2$$
by Lemma \ref{Lem: n even, p equiv n+1},
$$K_0(C_{2,6}^A)=\mathbb{Z}^2 $$
by Lemma \ref{Lem: n even, p equiv 0}.
\end{example}
\subsection{The Grothendieck group of $\C_{n,p}^D$}
Throughout this section consider the quiver $D_n$ to be linearly orientated as Figure \ref{quiverdn}. The  Auslander-Reiten quiver of $D^b(\mod (D_n))$ is shown in Figure \ref{o}.
\begin{figure}[H]
	\centering
	\begin{tikzpicture}[scale = 1.35]
		\foreach[count = \i] \txt in {1,...,5}
		\draw (-6.9 + 2*\i,0) -- (-6 + 2*\i,1.9) -- (-4.1 + 2*\i,1.9)-- (-5 + 2*\i, 0) -- (-6.9 + 2*\i,0);
		
		\node at (0.4,0.8) {\footnotesize $\mathcal{M}$};
		\node at (2.5,0.8) {\footnotesize $\Sigma \mathcal{M}$};
		\node at (4.6,0.8) {\footnotesize $\Sigma^2 \mathcal{M}$};
		\node at (-1.7,0.8) {\footnotesize $\Sigma^{-1} \mathcal{M}$};
		\node at (-3.8,0.8) {\footnotesize $\Sigma^{-2} \mathcal{M}$};
		
		\node at (-5,1) {$\cdots$};
		\node at (6,1) {$\cdots$};
	\end{tikzpicture}
	\caption{The Auslander-Reiten quiver of $D^b(\mod (D_n))$, where each component is equivalent to the Auslander-Reiten quiver of $\mathcal{M} = \mod (D_n)$.}
\label{o}
\end{figure}

We write the Auslander-Reiten quiver of $\mathcal{M} = \mod (D_n)$ by using coordinate system, see Figure \ref{ARDynkinD}.
\begin{figure}[H]
\centering
\resizebox{\textwidth}{!}{
$\xymatrix@R=10pt@C=8pt{ 
     &&&& (1,0)  \ar[dr]&& \cdots \ar[dr]&& (n-2,0)\ar[dr]&&(n-1,0) \\
     &&& (1,2)\ar[ur]\ar[r]\ar[dr]&(1,1)\ar[r]&\cdots\ar[ur]\ar[dr]\ar[r] & \cdots \ar[r]& (n-2,2)\ar[ur]\ar[dr]
     \ar[r] & (n-2,1)\ar[r]&(n-1,2)\ar[ur]\ar[r]&(n-1,1) \\
     &&\cdots \ar[ur]\ar[dr]&& \cdots \ar[ur]\ar[dr]&& \cdots\ar[ur]\ar[dr] && \cdots \ar[ur]\\
     & (1,n-2) \ar[dr] \ar[ur] & & \cdots  \ar[ur]\ar[dr]&& (n-2,n-2) \ar[ur]\ar[dr]&&(n-2,n-1)\ar[ur]\\
   (1,n-1)\ar[ur]&&\cdots \ar[ur]&& (n-2,n-1) \ar[ur]&& (n-1,n-1)\ar[ur] \\
}$}
\caption{The Auslander--Reiten quiver of $\mathrm{mod}(D_n)$.}
\label{ARDynkinD}
\end{figure}
The projective objects are $P_j=(1,j)$ for $0\leq j\leq n-1$, and the simple objects are $S_i=(n-i,n-1)$ for $2\leq i\leq n-1$, $S_1=(n-1,0)$ and $S_0=(n-1,1)$ for odd $n$; $S_1=(n-1,1)$ and $S_0=(n-1,0)$ for even $n$.

Throughout this section, we take the basis of $K_0(D^b(\mod (D_n)))$ to be
$$\{[S_0], [S_1], \ldots, [S_{n-1}]\},$$ where $S_i \in \mod (D_n) \subset D^b(\mod (D_n))$ corresponds to the simple module at vertex $i \in Q_0$.
Whenever we write $[S_i]$, we shall always consider $i \in \mod (n)$.
We note the classical result on the image of the projective objects in $K_0(\mod (D_n))$
\[ [P_i]=\left\{
\begin{array}{ll}
\sum_{j=i}^{n-1} [S_j] & \text{if\;}  i \geq 2,\\[3mm]
		\sum_{j=2}^{n-1} [S_j]+[S_{i}]& \text{if\;} i=0,1,
\end{array}
\right.
\]	
as well as the analogue for injective objects
\[ [I_i]=\left\{
\begin{array}{ll}
\sum_{j=0}^{i} [S_j] & \text{if\;} i \geq 2,\\[3mm]
       {[S_i]} & \text{if\;} i=0,1.
\end{array}
\right.
\]
	
\begin{lemma}\label{simple}
	Let $s = i + (n-1)t$ for $1 \leq i \leq n-1$. Then
\[ [(s,n-1)]=\left\{
\begin{array}{ll}
(-1)^t [S_{n-i}] & \text{if\;} 1 \leq i < n-1,\\[2mm]
		(-1)^t \sum_{i=0}^{n-1} [S_i] & \text{if\;} i=n-1.
\end{array}
\right.
\]	

\end{lemma}

\begin{proof}
Note that $S_j=(n-j,n-1)$ for $2\leq j\leq n-1$,
$$(n-1,n-1) \cong I_{n-1},~[I_{n-1}]=\sum_{i=0}^{n-1} [S_i],$$
and the natural isomorphism $\tau^{n-1} \xrightarrow{\sim}  \Sigma^{-1}$. If $1 \leq i < n-1$, then
$$[(s,n-1)]=[\tau^{-(n-1)t}S_{n-i}]=[\Sigma^{-t}S_{n-i}]=(-1)^t [S_{n-i}].$$
If $i=n-1$, then
$$[(s,n-1)]=[\tau^{-(n-1)t}I_{n-1}]=[\Sigma^{-t}I_{n-1}]=(-1)^t \sum_{i=0}^{n-1} [S_i].$$
\end{proof}

Now we consider the action of $\phi^p$ on simple objects $S_0$ and $S_1$.

\begin{lemma}\label{simple1}
Let $S_a\in\mod(D_n)$ be the simple object associated to the vertex $a\in\{0,1\}$. Then
$$\phi^p[S_a]=[\tau^pS_a]=\sum_{j=2}^{p+1} [S_j]+[S_{a+p}]$$
in $K_0(D^b(\mod(D_n)))$, where $p=1,\cdots,n-2$ and we consider $a+p$ modulo 2.
\end{lemma}
\begin{proof}
We prove the claim via induction. Note that there is an isomorphism
$$K_0(D^b(\mod(D_n)))\cong K_0(\mod(D_n)).$$
Let $p=1$ and consider the projective resolution
$$0\rightarrow P_2\rightarrow P_a\rightarrow S_a\rightarrow 0.$$
Then we have the following exact sequence
$$0\rightarrow \tau S_a\rightarrow I_2\rightarrow I_a\rightarrow 0,$$
so $[\tau S_a]=[I_2]-[I_a]=[S_2]+[S_{a+1}]$, where we consider $a+1$ modulo 2.

Now suppose that our claim is true for $p=k<n-2$, and
$$[\tau^k S_a]=\sum_{j=2}^{k+1} [S_j]+[S_{a+k}].$$
Observe that $[\tau^k S_a]=[P_{a+k}]-[P_{k+2}]$, as $KD_n$ is an hereditary algebra, this implies that we have the projective resolution
$$0\rightarrow P_{k+2}\rightarrow P_{a+k}\rightarrow \tau^k S_a\rightarrow 0.$$
Again, by classical Auslander-Reiten theory, we have the short exact sequence
$$0\rightarrow \tau^{k+1} S_a\rightarrow I_{k+2}\rightarrow I_{a+k}\rightarrow 0,$$
so we find
$$[\tau^{k+1} S_a]=[I_{k+2}]-[I_{a+k}]=\sum_{j=0}^{k+2} [S_j]-[I_{a+k}]=\sum_{j=2}^{k+2} [S_j]-[I_{a+k+1}].$$
\end{proof}

\begin{proposition}\label{DLem: p,q equiv}
	Let $p,q \geq 1$ with $p \equiv q \mod (2(n-1))$.
Then $K_0(C_{n,p}^D) \cong K_0(C_{n,q}^D)$.
\end{proposition}

\begin{proof}
Since $p \equiv q \mod (2(n-1))$,  $p = q +2a(n-1)$ for some $a \in \mathbb{Z}$. Then $p$ is even/odd if and only if $q$ is even/odd.
We have
$$
\begin{aligned}
(1\pm \phi^p)[X] &= [X] \pm [\tau^p X] = [X] \pm [\tau^{q + 2a(n-1)}X]\\
 &= [X] \pm [\Sigma^{-2a} (\tau^q X)] = [X] \pm [\tau^q X] \\
 &= (1 \pm \phi^q)[X]
\end{aligned}
$$
for all $X \in D^b(\mod {D_n})$.
Therefore $\mathrm{Im}(1 \pm \phi^p) = \mathrm{Im}(1 \pm \phi^q)$, so $K_0(C_{n,p}^D) \cong K_0(C_{n,q}^D)$ whenever $p \equiv q \mod (2(n-1))$.
\end{proof}

\begin{proposition}\label{DLem: p equiv 0}
	Let $p \equiv 0 \mod (2(n-1))$.
	Then $K_0(C_{n,p}^D) \cong \mathbb{Z}^n$.
\end{proposition}

\begin{proof}
	Since $p \equiv 0 \mod (2(n-1))$,  $p =2a(n-1)$ is even for some $a \in \mathbb{Z}$.
	Then
$$(1 - \phi^p) [X] = [X] - [\tau^p X] = [X] - [\tau^{2a(n-1)} X]
= [X] - [\Sigma^{-2a} X] =0$$ for all $X \in D^b(\mod {D_n})$.
	Hence $\mathrm{Im}(1 - \phi^p) =0$, and so $K_0(C_{n,p}^D) \cong K_0(\D^b(\mod {D_n})) \cong \mathbb{Z}^n$.
\end{proof}

\begin{proposition}
	Let $p \equiv n-1 \mod (2(n-1))$.
	Then
	\[ K_0(C_{n,p}^D) =\left\{
	\begin{array}{ll}
		\mathbb{Z}^n & \text{if\;$n$\;is\;odd},\\[2mm]
		(\mathbb{Z}/2\mathbb{Z})^n & \text{if\;$n$\;is\;even}.
	\end{array}
	\right.
	\]	
\end{proposition}

\begin{proof}
	Via the natural isomorphism $\tau^{n-1} \xrightarrow{\sim} \Sigma^{-1}$, we have
	\[
	(1 - \phi^{(2a-1)(n-1)})[S_i] = [S_i] \pm [\tau^{(2a-1)(n-1)}S_i] = [S_i] \pm [\Sigma^{(1-2a)}S_i] = [S_i] \mp [S_i].
	\]
	Hence,
	\[
	(1 - \phi^{(2a-1)(n-1)})[S_i] = 2[S_i], \qquad (1 + \phi^{(2a-1)(n-1)})[S_i] = 0.
	\]
	And so we find that,
	\[ K_0(C_{n,p}^D) =\left\{
	\begin{array}{ll}
		\mathbb{Z}^n & \text{if\;$n$\;is\;odd},\\[2mm]
		(\mathbb{Z}/2\mathbb{Z})^n & \text{if\;$n$\;is\;even}.
	\end{array}
	\right.
	\]	
\end{proof}

\begin{proposition}\label{DLem: p equiv 1}
	Let  $p \equiv 1 \mod (2(n-1))$. Then
\[ K_0(C_{n,p}^D) =\left\{
\begin{array}{ll}
\mathbb{Z} & \text{if\;$n$\;is\;odd},\\[2mm]
		\mathbb{Z}^2 & \text{if\;$n$\;is\;even}.
\end{array}
\right.
\]	
\end{proposition}

\begin{proof}
Note that $p$ is odd in this case. By Lemma \ref{simple} and Lemma \ref{simple1}, we have the set of relations
	\[
	\{[S_2] + [S_3], [S_3] + [S_4], \ldots, [S_{n-2}] + [S_{n-1}], \sum_{i=0}^{n-2} [S_i], [S_1]+[S_2]+[S_0]\}.
	\]
	The first $n-3$ relations in the above set imply that $[S_i] = (-1)^{i}[S_2]$ in $K_0(C_{n,p}^D)$ for $3\leq i\leq n-1$.

	 For the last 2 relations in the above set, we have the following two cases:
\begin{itemize}
  \item [(1)] If $n$ is odd, then the last 2 relations in the above set imply that $[S_0]+[S_1]=0$ and $[S_1]+[S_2]+[S_0]=0$. This means
  $$K_0(C_{n,p}^D)=\langle[S_1]\rangle\cong\mathbb{Z}.$$
  \item [(2)]If $n$ is even, then the last 2 relations in the above set imply that $[S_0]+[S_1]+[S_{n-2}]=0$ and $[S_1]+[S_2]+[S_0]=0$. This means
  $$K_0(C_{n,p}^D)=\langle[S_1],[S_0]\rangle\cong\mathbb{Z}^2.$$
\end{itemize}
\end{proof}

Note that when $p=1$, we recover the Grothendieck group of the classical cluster category of type $D_n$. This is a result originally due to Barot, Kussin and Lenzing \cite{BKL}, see Lemma \ref{36}.
\vspace{3mm}

\section*{Acknowledgement}
We would like to thank the referees for carefully reading our manuscript and for pointing out problems in the previous version, as well as for their helpful suggestions.
\vspace{3mm}

\hspace{-4mm}\textbf{Data Availability}\hspace{2mm} Data sharing not applicable to this article as no datasets were generated or analysed during
the current study.
\vspace{2mm}

\hspace{-4mm}\textbf{Conflict of Interests}\hspace{2mm} The authors declare that they have no conflicts of interest to this work.

\end{document}